\documentclass[11pt,reqno]{amsart} 
\numberwithin{equation}{section} 

\usepackage{amsmath,amsthm}           
\usepackage{amssymb}                  
\usepackage{mathtools}                
\mathtoolsset{showonlyrefs}          
\usepackage{a4wide}                   

\usepackage{enumitem}                

\usepackage{stmaryrd}                
\usepackage{xfrac}                   
\usepackage{faktor}                  
\usepackage{physics}                 
\usepackage{graphicx,scalerel}
\newcommand{\gad}{\mathrel{\rotatebox[origin=c]{180}{$\dag$}}}

\usepackage[colorlinks]{hyperref}   

\usepackage{subcaption}             

\usepackage{tabstackengine}         
\setstackEOL{;}                     
\setstackTAB{,}                     
\setstacktabbedgap{1ex}             
\setstackgap{L}{1.0\normalbaselineskip} 

\usepackage{graphicx}
\usepackage{tikz}                    
\usepackage{tikz-cd}                
\usepackage{pgfplots}               
\pgfplotsset{compat=1.18}           

\usetikzlibrary{arrows.meta, matrix, backgrounds, patterns, fadings}
\usetikzlibrary{arrows, decorations.pathmorphing, calc}
\usetikzlibrary{shapes, automata, petri, positioning}

\tikzset{
	place/.style={circle, thick, draw=black, fill=gray!50, minimum size=20mm},
	state/.style={circle, thick, draw=blue!75, fill=blue!20, minimum size=20mm},
	cross/.pic = {
		\draw[rotate = 45] (-0.2,0) -- (0.2,0);
		\draw[rotate = 45] (0,-0.2) -- (0, 0.2);
	}
}

\newcommand{\inner}[1]{\langle #1\rangle}
\newcommand{\floor}[1]{\lfloor #1 \rfloor}
\newcommand{\ceil}[1]{\lceil #1 \rceil}

\newcommand{\eps}{\epsilon}
\newcommand{\rmd}{\mathrm{d}}   
\newcommand{\ndef}[1]{{\textcolor{blue}{#1}}}

\newcommand{\z}{\mathbf{z}}

\renewcommand{\r}{\mathbf{r}}

\newcommand{\T}{\mathbb{T}}
\newcommand{\where}{\quad\text{where}\quad}

\theoremstyle{plain}
\newtheorem{Th}{Theorem}[section]
\newtheorem{Lemma}[Th]{Lemma}
\newtheorem{Cor}[Th]{Corollary}

\theoremstyle{definition}
\newtheorem{Def}[Th]{Definition}

\newtheorem{Rem}[Th]{Remark}
\newtheorem{?}[Th]{Problem}

\title[]
{Hydrodynamic Limit with a Weierstrass-type result}

\author{Gabriel S. Nahum$^{\gad}$}
\email{{\tt gabriel.nahum@protonmail.com}}

\usepackage{wrapfig}
\begin{document}

\begin{abstract}
We show that any positive, continuous, and bounded function can be realised as the diffusion coefficient of an evolution equation associated with a gradient interacting particle system. The proof relies on the construction of an appropriate model and on the entropy method.
\end{abstract}

\maketitle

\renewcommand{\thefootnote}{} 
\footnote{$^{\gad}$\textit{E-mail address}:
    \texttt{gabriel.nahum@protonmail.com}}
\renewcommand{\thefootnote}{\arabic{footnote}}

\tableofcontents

\section{Introduction}

Understanding the macroscopic behaviour of interacting particle systems is a central theme in probability and statistical mechanics. This work studies the local density of particles in a discrete Markovian system, which, after a rescaling in space and time, converges to the solution of a particular differential equation. From a probabilistic viewpoint, this corresponds to a law of large numbers. This passage from the discrete to the continuum, or from the microscopic to the macroscopic, is known in the physics literature as the hydrodynamic limit. Within the landscape of interacting particle systems, gradient systems play an important role, as their local structure makes them mathematically very tractable. These are models in which the algebraic current can be expressed as the discrete gradient of a potential \cite[II, Subsection 2.4]{spohn:book}, further simplifying hydrodynamical analysis when compared to non-gradient models. For example, for a variety of models both equilibrium and non-equilibrium fluctuations can be obtained \cite{BCGPJS22,GJMRM22,GJMN20,GHS25} as opposed to non-gradient models where this remains largely an open problem (see discussion in \cite{PHD:clement} and references therein). Moreover, {in general the diffusion coefficient is described with Green Kubo formula \cite[II, Subsection 2.2]{spohn:book}, and particularising} for gradient systems one can {compute it} explicitly with the knowledge of the invariant measures. This is convenient in devising models associated with a target diffusion. A physical interpretation has also been proposed for these systems, pointing to a dynamical term in the description of the general diffusivity that vanishes if \cite[II, Subsection 2.2]{spohn:book} and only if \cite{makiko} the model is of gradient type, thus representing a purely static behaviour. It is therefore surprising {that few general} results within this class are available. The present work is a first attempt to address the fundamental question of how flexible such models can be: given a target diffusion equation, can one construct a microscopic dynamics whose empirical density evolves accordingly? It is worth mentioning that generality of functions has been investigated in the context of boundary conditions \cite{mangi}, allowing for generic polynomial functions via prescribed boundary dynamics evolving at the scale of the current. {An analysis of the dynamics in the bulk,} as in the present work, becomes in some aspects more challenging due to the requirement of maintaining the gradient property. {In this way, we shall restrict the studies to} the one-dimensional discrete torus.

To provide some additional context regarding possible diffusivities, a classical example of linear diffusivity arises from the Symmetric Simple Exclusion Process \cite{KL:book}, arguably the simplest dynamics for which much is known. In this model, particles perform nearest-neighbour random walks with no directional bias in a discrete lattice and subject to an exclusion rule that prevents multiple occupancy of sites, yielding a Markov process with conserved density. {Gradient models can encapsulate polynomial diffusivity, with the Porous Media Model (PMM) \cite{GLT} acting as a toy model} for diffusion coefficients of the form $D(\rho)=\rho^m$, for positive integer values of $m$. The dynamics is easily described: under the exclusion rule, a pair composed by a vacant and occupied sites exchange their occupation values with rate given by the number of boxes composed by $m$ aligned particles around the pair. {The derivation of non-polynomial diffusion from exclusion type dynamics was addressed in \cite{s2f}, where it was performed an extension of the dynamics to non-integer values of the parameter $m$ via a generalised binomial expansion, successful for values $m\in(-1,1)$.} This allows a microscopic description of the transition from slow to fast diffusion, continuously parametrised, by a collection of gradient models. {It is relevant to mention that the two-parts works \cite{FIS97,ES96} identify the diffusivity $D(\rho)=\rho^m$ for $m\geq 1$ for processes that are continuous in nature, associated with coupled oscillators and where other tools apply.}

{The work \cite{CGN24} follows on the direction of investigating non-polynomial diffusion in more generality and with long-range interactions. The works \cite{s2f,CGN24} evidence a deficiency in the porous media collection in generating particular diffusion, yielding possibly "negative rates" by definition, that was then addressed in \cite{N25}. In it, the author introduces a collection of gradient dynamics modelling the Bernstein polynomial basis and generalising the PMM.  Linear combinations in this basis, in the same spirit as \cite{s2f}, were analysed in \cite{N25}, allowing for the introduction of the complete PMM family, continuously parametrised and yielding the diffusivity $D(\rho)=\rho^m$ for any $m>1$, thus resolving the main open problem posed in \cite{s2f}.} 
The aforementioned works paved the way for the analysis carried out in the present text, culminating in a first Weierstrass-type result.

The model introduced here has critical differences from those studied in \cite{N25}, and requires the introduction of some novel arguments. The main strategy for proving the hydrodynamic limit, however, is analogous, as we follow the well established entropy method \cite{GPV}. This choice of methodology is justified both by the possibility of relying on previous and convenient results, and by the fact that it prepares the ground for a future complete open-boundary analysis. With this in mind, in the next sections we present only the main model and novel results, directing the reader to appropriate references for specific steps of the proof. {We highlight the need for novel arguments due to the non-uniform continuity of the microscopic potential associated with the gradient property (see Lemmas \ref{lem:h-treat} and \ref{lem:rep-h} and their applications in the proof of Theorem \ref{th:hydro}). This is the main technical challenge with respect to \cite{N25}}.

At a more technical level, and pointing to future work and open problems, the restriction to \textit{positive} diffusivities is crucial, as it ensures that the model is irreducible. 

\subsection{Open problems and future work}

{Our main model is irreducible, and the hopping rates depend on a “mesoscopic” part of the configuration: letting $N>> 1$ be the length of the system, the constraints are dependent of the information in a ball of radius $\ell_N+1=o(N)$, where $\ell_N\xrightarrow{}\infty$ as $N\xrightarrow{}\infty$. If the function $\beta$ were allowed to take the value zero, the resulting dynamics would belong to the class of \textit{non-cooperative} \cite{AS24} models, for which "mobile clusters" exist. Mobile clusters are blocks of configurations that can move together and facilitate the transport of mass thorough the system. These objects play a central phenomenological and technical role in non-irreducible systems, associated with diffusivities that can attain trivial value for specific densities \cite{GLT,BDGN,s2f,N25,LR:CDG2022,CGN24}. In the present setting, however, allowing diffusivities that attain zero would give rise to mobile clusters whose length depends on the size of the system, making them too large to be analysed via standard arguments. This is in contrast with \cite{N25}, where their size is independent of the length of the system and therefore of completely local nature. The nature of these clusters brings the system technically closer to a \textit{cooperative} model, that are we known to be non-gradient \cite{AS24}. {In this sense, it may act as a bridge between cooperative and non-cooperative models. This is left for future work. } 

We also identified a phase transition related with the range of interaction $2\ell_N+1$, when $\ell_N=O(N)$. The author is currently investigating the hydrodynamics of the resulting system in this regime, when in contact with reservoirs. One may also question the extension to higher dimensions. This can be performed straightforwardly by letting the process evolve as a one-dimensional model along each coordinate direction, yielding a higher-dimensional hydrodynamic equation. A challenging question concerns the derivation of fractional hydrodynamic equations, associated with long-range models. {This would require a ``long-jumps" extension of the nearest-neighbour dynamics in \cite{N24}, for which the arguments in \cite{LR:CDG2022} violate the gradient property.}

\subsection{Outline of the paper}
In Section \ref{sec:setup}, we introduce the discrete setup and the main process (Definition \ref{def:model}), and we discuss its gradient property (Lemma \ref{lem:grad}) along with relations to other models. The main result, Theorem \ref{th:hydro}, is stated in Subsection \ref{subsec:hl}. Its proof is performed in Section \ref{sec:hl} using the entropy method, which consists of three steps: tightness, characterization of limit points, and energy estimate. In this work, we focus primarily on the second step, providing references and justifications for the remaining steps, which is performed in Subsection \ref{subsec:scheme}. Technical approximations, the so-called ''replacement lemmas", specific to our dynamics and novel, are proved in Section \ref{sec:rep-lem}, while auxiliary replacement lemmas from the literature are collected in Appendix \ref{app:replacement} for completeness and internal references.

\section{Setup and Main result}\label{sec:setup}
Let \setcounter{footnote}{0}\footnote{Throughout this text, for the reader’s convenience, newly introduced notation and objects outside mathematical environments will be highlighted in \ndef{blue}.
} $ \ndef{\mathbb{N}_+} $ be the set of positive natural numbers and denote by $ \ndef{N} \in\mathbb{N}_+ $ a scaling parameter. Our discrete lattice is $ \ndef{\mathbb{T}_N} $, the one dimensional discrete torus, $ \mathbb{T}_N=\{1,\dots, N\} $ with the identification $ 0\equiv N $.  For any $ x < y\in\mathbb{Z} $, that can be viewed as elements in $\T_N$ by considering their standard projections, we define the discrete interval $ \ndef{\llbracket x,y\rrbracket} $, composed by all the points between $ x,y $ (including $ x,y $) in $ \mathbb{T}_N $, where the order has been inherited from the one in $\mathbb Z$.

The central dynamics of this work is an interacting particle system, following a Markovian law, satisfying the exclusion rule and situated on the discrete torus $\T_N$. A configuration of particles is an element of the state space $\ndef{\Omega_N}:=\{0,1\}^{\mathbb{T}_N}$ and will be recurrently denoted by the letters $\eta$ and $\xi$. We denote by $\ndef{\eta(x)}\in\{0,1\}$ the occupation value of $\eta\in\Omega_N$ at the site $x\in\mathbb{T}_N$. 

In the next subsection we introduce the main process and relevant properties, and in subsection \ref{subsec:analysis} we present the topological setting. Subsection \ref{subsec:hl} is devoted to presenting the main result, that will be shown in Section \ref{sec:hl}.

\subsection{The model}\label{subsec:model}
We start by introducing the relevant operators associated with the process.  
\begin{Def}
In what follows, let $\eta\in\Omega_N$ , $f:\Omega_N\to\mathbb{R}$ and $x,y,z\in\mathbb{T}_N$ be arbitrary:
\begin{itemize}
        \item Let $\ndef{\pi_x}:\Omega_N\to\{0,1\}$ be the projection $\pi_x(\eta)=\eta(x)$;
        
        \item The shift operator, $\ndef{\tau}:\eta\mapsto\tau\eta$, is defined through $\pi_x(\tau\eta)=\pi_{x+1}(\eta)$. We short-write $\ndef{\tau^i}=\circ_{j=1}^i\tau$ for its $i$-th composition;

        \item Denote by $\ndef{\theta_{x,y}}$ the operator that exchanges the occupation-value of the sites $x,y$,
        \begin{align}
            (\theta_{x,y}\eta)(z)
            =\eta(z)1_{z\neq x,y}
            +\eta(y)1_{z=x}
            +\eta(x)1_{z=y};
        \end{align}

        \item For $\mathcal{O}=\tau,\theta_{x,y}$, let $\ndef{\mathcal{O} f(\eta)}:=f(\mathcal{O}\eta)$.
    \end{itemize}
    Moreover, introduce the linear operators $\ndef{\nabla_{x,y}},\ndef{\nabla},\ndef{\Delta}$ through 
    \begin{align}
        \nabla_{x,y}[f]:=\theta_{x,y}f-f
        ,\quad
        \nabla[f]:=\tau f-f 
        \quad\text{and}\quad
        \Delta[f]=\nabla^2[\tau^{-1}f]
        .
    \end{align}
    Note that $\nabla$ corresponds to the forward difference operator and $\Delta$ to the discrete Laplacian operator.
\end{Def}

We are ready to introduce the model at the core of this paper.
\begin{Def}\label{def:model}
Let $\ndef{\beta}$ be a fixed continuous function in $\T$ and such that $0< \beta(u)\leq 1$ for each $u\in\T$. Fix $0<\ndef{\ell_N}=o(N)$ such that $\ell_N\to+\infty$ as $N\to+\infty$. We will study the time-scaled Markov generator $\ndef{\mathfrak{L}_N}$, given by
\begin{align}
    \mathfrak{L}_N
    :=N^2\mathcal{L}^{\beta,\ell_N}, 
\end{align}
The generator $\ndef{\mathcal{L}^{\beta,\ell_N}}$ is defined, for each $f:\Omega_N\to\mathbb{R}$, through 
\begin{align}\label{gen:r}
    \mathcal{L}^{\beta,\ell_N}[f]
    &:=\sum_{x\in\T_N}
    \tau^x\mathbf{c}_{\beta,\ell_N}
    \Big\{
        \tau^x\mathbf{e}_{0,1}
        +
        \tau^x\mathbf{e}_{1,0}
    \Big\}
    \nabla_{x,x+1}[f],
\end{align}
where the \textit{exclusion constraint} for a hopping to the right is given by $\mathbf{e}_{0,1}$, while to the left, $\mathbf{e}_{1,0}$; and $\tau^x\mathbf{c}_{\beta,\ell_N}$ the main dynamical constraint acting in the node $\{x,x+1\}$. They are precisely given, for each $\eta\in\Omega_N$, by $\textcolor{blue}{\mathbf{e}_{0,1}}(\eta)=\eta(0)(1-\eta(1))+(1-\eta(0))\eta(1)$ , $\textcolor{blue}{\mathbf{e}_{1,0}(\eta)}=(1-\eta(0))\eta(1)$; and $\textcolor{blue}{\mathbf{c}_{\beta,\ell_N}}>0$ is defined in terms of \textit{windows of interaction}: for each $0\leq j\leq \ell_N$, let $\ndef{W_j}:=\llbracket -j,-j+\ell_N+1\rrbracket\backslash\{0,1\}$ and shorten $\ndef{\inner{\eta}_{j}}\equiv\tfrac{1}{\ell_N}\sum_{x\in W_j}\eta(x)$ for the density in the box $W_j$. Then,
\begin{align}\label{constraint}
    \mathbf{c}_{\beta,\ell_N}(\eta)
    :=
    \frac{1}{\ell_N+1}
    \sum_{j=0}^{\ell_N}
    \beta(\inner{\eta}_{j})
    .
\end{align}
The dynamics is illustrated in Figure \ref{fig:dynamics}. It is convenient to shorten the rates
\begin{align}\label{rate-j}
    \mathbf{r}_{\beta,\ell_N}^j(\eta):=(\mathbf{e}_{0,1}(\eta)+\mathbf{e}_{1,0}(\eta))\beta(\inner{\eta}_{j})
    \quad\text{and}\quad
    \mathbf{r}_{\beta,\ell_N}
    :=
    \frac{1}{\ell_N+1}
    \sum_{j=0}^{\ell_N}
    \mathbf{r}_{\beta,\ell_N}^j
    .
\end{align}
\end{Def}
By translation invariance, the constraint for an exchange at a node $\{x,x+1\}$, with generic $x\in\T_N$, is defined via translation, with rate coinciding with $\tau^x\mathbf{r}_{\beta,\ell_N}$

\begin{wrapfigure}{r}{0.65\textwidth}		
\centering
\tikzset{
pattern size/.store in=\mcSize, 
pattern size = 5pt,
pattern thickness/.store in=\mcThickness, 
pattern thickness = 0.3pt,
pattern radius/.store in=\mcRadius, 
pattern radius = 1pt}\makeatletter
\pgfutil@ifundefined{pgf@pattern@name@_6l8lnutyt}{
\pgfdeclarepatternformonly[\mcThickness,\mcSize]{_6l8lnutyt}
{\pgfqpoint{-\mcThickness}{-\mcThickness}}
{\pgfpoint{\mcSize}{\mcSize}}
{\pgfpoint{\mcSize}{\mcSize}}
{\pgfsetcolor{\tikz@pattern@color}
\pgfsetlinewidth{\mcThickness}
\pgfpathmoveto{\pgfpointorigin}
\pgfpathlineto{\pgfpoint{\mcSize}{0}}
\pgfpathmoveto{\pgfpointorigin}
\pgfpathlineto{\pgfpoint{0}{\mcSize}}
\pgfusepath{stroke}}}
\makeatother

 
\tikzset{
pattern size/.store in=\mcSize, 
pattern size = 5pt,
pattern thickness/.store in=\mcThickness, 
pattern thickness = 0.3pt,
pattern radius/.store in=\mcRadius, 
pattern radius = 1pt}
\makeatletter
\pgfutil@ifundefined{pgf@pattern@name@_38k0h76s6}{
\pgfdeclarepatternformonly[\mcThickness,\mcSize]{_38k0h76s6}
{\pgfqpoint{-\mcThickness}{-\mcThickness}}
{\pgfpoint{\mcSize}{\mcSize}}
{\pgfpoint{\mcSize}{\mcSize}}
{
\pgfsetcolor{\tikz@pattern@color}
\pgfsetlinewidth{\mcThickness}
\pgfpathmoveto{\pgfpointorigin}
\pgfpathlineto{\pgfpoint{0}{\mcSize}}
\pgfusepath{stroke}
}}
\makeatother

 
\tikzset{
pattern size/.store in=\mcSize, 
pattern size = 5pt,
pattern thickness/.store in=\mcThickness, 
pattern thickness = 0.3pt,
pattern radius/.store in=\mcRadius, 
pattern radius = 1pt}
\makeatletter
\pgfutil@ifundefined{pgf@pattern@name@_qcq731b7b lines}{
\pgfdeclarepatternformonly[\mcThickness,\mcSize]{_qcq731b7b}
{\pgfqpoint{0pt}{0pt}}
{\pgfpoint{\mcSize+\mcThickness}{\mcSize+\mcThickness}}
{\pgfpoint{\mcSize}{\mcSize}}
{\pgfsetcolor{\tikz@pattern@color}
\pgfsetlinewidth{\mcThickness}
\pgfpathmoveto{\pgfpointorigin}
\pgfpathlineto{\pgfpoint{\mcSize}{0}}
\pgfusepath{stroke}}}
\makeatother

 
\tikzset{
pattern size/.store in=\mcSize, 
pattern size = 5pt,
pattern thickness/.store in=\mcThickness, 
pattern thickness = 0.3pt,
pattern radius/.store in=\mcRadius, 
pattern radius = 1pt}
\makeatletter
\pgfutil@ifundefined{pgf@pattern@name@_gas5ov9zb lines}{
\pgfdeclarepatternformonly[\mcThickness,\mcSize]{_gas5ov9zb}
{\pgfqpoint{0pt}{0pt}}
{\pgfpoint{\mcSize+\mcThickness}{\mcSize+\mcThickness}}
{\pgfpoint{\mcSize}{\mcSize}}
{\pgfsetcolor{\tikz@pattern@color}
\pgfsetlinewidth{\mcThickness}
\pgfpathmoveto{\pgfpointorigin}
\pgfpathlineto{\pgfpoint{\mcSize}{0}}
\pgfusepath{stroke}}}
\makeatother

 
\tikzset{
pattern size/.store in=\mcSize, 
pattern size = 5pt,
pattern thickness/.store in=\mcThickness, 
pattern thickness = 0.3pt,
pattern radius/.store in=\mcRadius, 
pattern radius = 1pt}
\makeatletter
\pgfutil@ifundefined{pgf@pattern@name@_6q9qiikr6}{
\makeatletter
\pgfdeclarepatternformonly[\mcRadius,\mcThickness,\mcSize]{_6q9qiikr6}
{\pgfpoint{-0.5*\mcSize}{-0.5*\mcSize}}
{\pgfpoint{0.5*\mcSize}{0.5*\mcSize}}
{\pgfpoint{\mcSize}{\mcSize}}
{
\pgfsetcolor{\tikz@pattern@color}
\pgfsetlinewidth{\mcThickness}
\pgfpathcircle\pgfpointorigin{\mcRadius}
\pgfusepath{stroke}
}}
\makeatother
\tikzset{every picture/.style={line width=0.75pt}} 

\begin{tikzpicture}[x=0.75pt,y=0.75pt,yscale=-1,xscale=1]

\draw  [color={rgb, 255:red, 0; green, 0; blue, 0 }  ,draw opacity=1 ][pattern=_6l8lnutyt,pattern size=6pt,pattern thickness=0.75pt,pattern radius=0pt, pattern color={rgb, 255:red, 0; green, 0; blue, 0}][line width=1.5]  (348.95,112.18) -- (467.64,112.18) -- (467.64,145.18) -- (348.95,145.18) -- cycle ;
\draw  [color={rgb, 255:red, 0; green, 0; blue, 0 }  ,draw opacity=1 ][pattern=_38k0h76s6,pattern size=6pt,pattern thickness=0.75pt,pattern radius=0pt, pattern color={rgb, 255:red, 0; green, 0; blue, 0}][line width=1.5]  (145.5,221.18) -- (281.14,221.18) -- (281.14,254.18) -- (145.5,254.18) -- cycle ;
\draw    (125.5,290.91) -- (497.23,292.83) (159.52,287.09) -- (159.48,295.09)(193.52,287.27) -- (193.48,295.26)(227.52,287.44) -- (227.48,295.44)(261.52,287.62) -- (261.48,295.62)(295.52,287.79) -- (295.48,295.79)(329.52,287.97) -- (329.48,295.97)(363.52,288.14) -- (363.48,296.14)(397.52,288.32) -- (397.48,296.32)(431.52,288.49) -- (431.48,296.49)(465.52,288.67) -- (465.47,296.67) ;
\draw  [fill={rgb, 255:red, 155; green, 155; blue, 155 }  ,fill opacity=1 ] (281.14,270.68) .. controls (281.14,261.57) and (288.73,254.18) .. (298.09,254.18) .. controls (307.45,254.18) and (315.05,261.57) .. (315.05,270.68) .. controls (315.05,279.8) and (307.45,287.18) .. (298.09,287.18) .. controls (288.73,287.18) and (281.14,279.8) .. (281.14,270.68) -- cycle ;
\draw   (315.05,270.68) .. controls (315.05,261.57) and (322.64,254.18) .. (332,254.18) .. controls (341.36,254.18) and (348.95,261.57) .. (348.95,270.68) .. controls (348.95,279.8) and (341.36,287.18) .. (332,287.18) .. controls (322.64,287.18) and (315.05,279.8) .. (315.05,270.68) -- cycle ;
\draw  [color={rgb, 255:red, 0; green, 0; blue, 0 }  ,draw opacity=1 ][pattern=_qcq731b7b,pattern size=6pt,pattern thickness=0.75pt,pattern radius=0pt, pattern color={rgb, 255:red, 0; green, 0; blue, 0}][line width=1.5]  (348.95,188.18) -- (382.86,188.18) -- (382.86,221.18) -- (348.95,221.18) -- cycle ;
\draw  [color={rgb, 255:red, 208; green, 2; blue, 27 }  ,draw opacity=1 ][line width=1.5]  (281.14,254.18) -- (348.95,254.18) -- (348.95,287.18) -- (281.14,287.18) -- cycle ;
\draw  [color={rgb, 255:red, 0; green, 0; blue, 0 }  ,draw opacity=1 ][pattern=_gas5ov9zb,pattern size=6pt,pattern thickness=0.75pt,pattern radius=0pt, pattern color={rgb, 255:red, 0; green, 0; blue, 0}][line width=1.5]  (179.9,188.18) -- (281.14,188.18) -- (281.14,221.18) -- (179.9,221.18) -- cycle ;
\draw [pattern=_6q9qiikr6,pattern size=6pt,pattern thickness=0.75pt,pattern radius=0.75pt, pattern color={rgb, 255:red, 0; green, 0; blue, 0}] [dash pattern={on 4.5pt off 4.5pt}]  (145.5,254.18) -- (199.23,346.83) ;
\draw  [dash pattern={on 4.5pt off 4.5pt}]  (233.14,346.83) -- (281.14,254.18) ;
\draw  [dash pattern={on 4.5pt off 4.5pt}]  (215.23,108.5) -- (179.9,188.18) ;
\draw  [dash pattern={on 4.5pt off 4.5pt}]  (249.14,108.5) -- (281.14,188.18) ;
\draw  [dash pattern={on 4.5pt off 4.5pt}]  (348.95,112.18) -- (402.57,48.83) ;
\draw  [dash pattern={on 4.5pt off 4.5pt}]  (467.64,112.18) -- (436.48,48.83) ;
\draw  [dash pattern={on 4.5pt off 4.5pt}]  (249.14,108.5) -- (382.86,188.18) ;
\draw  [dash pattern={on 4.5pt off 4.5pt}]  (249.14,108.5) -- (348.95,188.18) ;

\draw (462.67,300.83) node [anchor=north west][inner sep=0.75pt]  [font=\tiny] [align=left] {$\displaystyle 5$};
\draw (405.73,262.08) node [anchor=north west][inner sep=0.75pt]    {$\cdots $};
\draw (202.73,262.08) node [anchor=north west][inner sep=0.75pt]    {$\cdots $};
\draw (182,300.33) node [anchor=north west][inner sep=0.75pt]  [font=\tiny] [align=left] {$\displaystyle N-3$};
\draw (428.67,300.33) node [anchor=north west][inner sep=0.75pt]  [font=\tiny] [align=left] {$\displaystyle 4$};
\draw (394.17,300.33) node [anchor=north west][inner sep=0.75pt]  [font=\tiny] [align=left] {$\displaystyle 3$};
\draw (359.67,300.33) node [anchor=north west][inner sep=0.75pt]  [font=\tiny] [align=left] {$\displaystyle 2$};
\draw (326.17,300.33) node [anchor=north west][inner sep=0.75pt]  [font=\tiny] [align=left] {$\displaystyle 1$};
\draw (292.67,300.33) node [anchor=north west][inner sep=0.75pt]  [font=\tiny] [align=left] {$\displaystyle 0$};
\draw (250.17,300.33) node [anchor=north west][inner sep=0.75pt]  [font=\tiny] [align=left] {$\displaystyle N-1$};
\draw (215,300.33) node [anchor=north west][inner sep=0.75pt]  [font=\tiny] [align=left] {$\displaystyle N-2$};
\draw (146.6,300.33) node [anchor=north west][inner sep=0.75pt]  [font=\tiny] [align=left] {$\displaystyle N-4$};
\draw (380,159.58) node [anchor=north west][inner sep=0.75pt]    {$\cdots $};
\draw (150,350.23) node [anchor=north west][inner sep=0.75pt]    {$W_{\ell_N}=\llbracket-\ell_N,-1\rrbracket$};
\draw (150,78.9) node [anchor=north west][inner sep=0.75pt]    {$W_{\ell_N-1}=\llbracket-\ell_N+1,-1\rrbracket\cup\{2\}$};
\draw (380,30) node [anchor=north west][inner sep=0.75pt]    {$W_{0}=\llbracket2,\ell+1\rrbracket$};
\draw (227.39,158.42) node [anchor=north west][inner sep=0.75pt]    {$\cdots $};
\end{tikzpicture}
\caption{The rate for an exchange at the node $\{0,1\}$ (box in \textcolor{red}{red}) depends on the density of particles in the windows $W_0,W_1,\dots,W_{\ell_N}$: it is given by the average, over these windows, of $\beta$ evaluated at the empirical average in the corresponding window, $\inner{\eta}_j$. Precisely, the exchange realises with rate $\frac{1}{\ell_N+1}
    \sum_{j=0}^{\ell_N}
    \beta(\inner{\eta}_{j})\rmd t$.}    \label{fig:dynamics}
\end{wrapfigure}
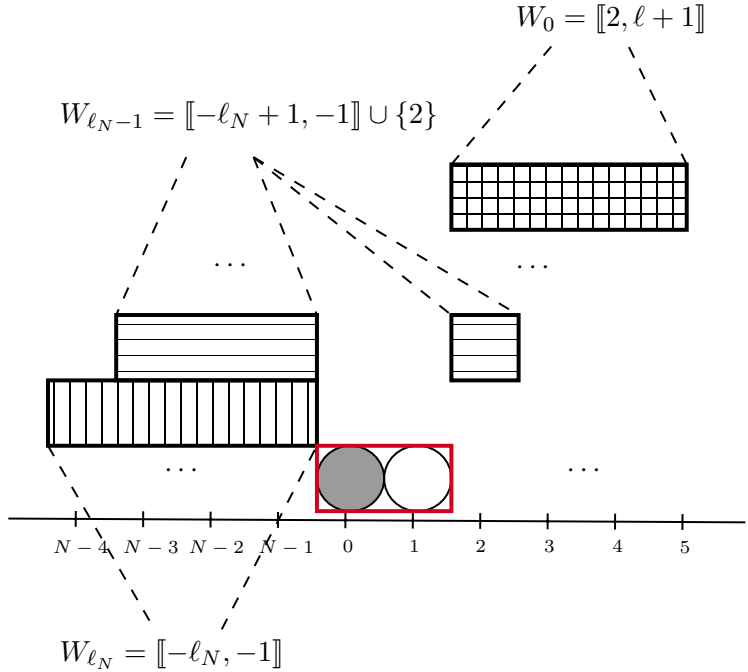

We make some remarks on the model just introduced.
\begin{Rem}
    Because the function $\beta$ is positive, the process $\mathcal{L}^{\beta,\ell_N}$ is irreducible. Moreover, being also symmetric and of gradient type, as we shall see shortly, the Bernoulli product measures with constant profile are invariant measures. Because $\beta$ is normalized, $\mathbf{c}_{\beta,\ell_N}\leq 1$. We shall make no further assumptions on $\beta$ beyond:
    \begin{itemize}
        \item boundedness;
        \item positivity;
        \item continuity.
    \end{itemize}
    
It is crucial to note that the process associated with $\mathcal{L}^{\beta,\ell_N}$ is a superposition of the elementary processes introduced in \cite{N24} and representing the Bernstein polynomial basis. For each $0\leq n\leq\ell_N$, consider the generator \ndef{$\mathcal{B}^{n,\ell_N}$}, given by replacing in \eqref{gen:r} the constraint $\mathbf{c}_{\beta,\ell_N}$ by $\mathbf{b}_{n,\ell_N}$, with 
\begin{align}\label{rate:b}
    \mathbf{b}_{n,\ell_N}
    :=\frac{1}{\ell_N+1}\sum_{j=0}^{\ell_N}
    \mathbf{b}_{n,\ell_N}^j
    \quad\text{and}\quad
    \mathbf{b}_{n,\ell_N}^j(\eta)
    :=\sum_{j=0}^{\ell_N}\mathbf{1}\{\inner{\eta}_{j}=\tfrac{n}{\ell_N}\}.
\end{align}
It holds that
\begin{align}\label{gen:decomp}
    \mathcal{L}^{\beta,\ell_N}
    =\sum_{n=0}^{\ell_N}
    \beta(\tfrac{n}{\ell_N})\mathcal{B}^{n,\ell_N}
    .
\end{align}

\end{Rem}
The constraint $\mathbf{c}_{\beta,\ell_N}$ should be thought of as a discrete analogous of the Bernstein polynomial of degree $\ell_N$, and the sequence $(\beta(\tfrac{n}{\ell_N}))_{n\geq 0}$ as the Bezier coefficients. Following this rationale, for future reference we introduce, for each integer $L\geq 0$, 
\begin{align}\label{binom}
    \text{B}_{\beta,L}:=
    \sum_{n=0}^{L}
    \beta(\tfrac{n}{L})
    \text{B}_{n,L}
    ,
    \quad\text{with}\quad
    \text{B}_{n,L}(\rho)
    =\binom{L}{n}\rho^n(1-\rho)^{L-n}
    ,\;\text{for any } \rho\in[0,1].
\end{align}

In \cite[Proposition 2.7]{N24} it is proved that the generator $\mathcal{B}^{n,\ell_N}$ is associated with a \textit{gradient model}, in the sense of \cite[II, Subsection 2.4]{spohn:book}. Thus, from the linear combination in \eqref{gen:decomp}, our main process inherits this property. This is the content of the next lemma, whose proof is direct given the expression in \cite[Proposition 2.7]{N24}.
\begin{Lemma}\label{lem:grad}
Let $\ndef{\mathbf{J}_\beta}$ be the algebraic current, defined through $\mathcal{L}^{\beta,\ell_N}[\pi_0]=\nabla\tau^{-1}[-\mathbf{J}_{\beta,\ell_N}]$, where $\mathbf{J}_{\beta,\ell_N}$ is given, explicitly, by
\begin{align}\label{curr-j}
    \mathbf{J}_{\beta,\ell_N}
    =
    \frac{1}{\ell_N+1}
    \sum_{j=1}^{\ell_N}
    \mathbf{J}_{\beta}^j(\eta)
\quad\text{with}\quad
    \mathbf{J}_{\beta}^j(\eta)
    =\beta(\inner{\eta}_{j})
    (
    \mathbf{e}_{0,1}(\eta)
    -
    \mathbf{e}_{1,0}(\eta)
    )
    .
\end{align}

For each $\eta\in\Omega_N$, let $\ndef{P_{\ell_N}}(\eta)$ be the number of particles in the box $\llbracket0,\ell_N\rrbracket$. It holds that
\begin{align}
    \mathbf{J}_{\beta,\ell_N}=-\nabla \mathbf{H}_{\beta,\ell_N},
\end{align}
where $\mathbf{H}_{\beta,\ell_N}=\mathbf{h}_{\beta,\ell_N}+\mathbf{g}_{\beta,\ell_N}$ and 
\begin{align}
\label{h}
    \mathbf{h}_{\beta,\ell_N}(\eta)
    &=
    \frac{1}{\ell_N+1}
    \sum_{n=0}^{P_{\ell_N}(\eta)-1}
    \beta(\tfrac{n}{\ell_N})
    ,
    \\
\label{g}
    \mathbf{g}_{\beta,\ell_N}(\eta)
    &=
    \frac{1}{\ell_N+1}
    \sum_{j=1}^{\ell_N}
    \sum_{i=1}^{j}
    \tau^{-i}\mathbf{J}_{\beta,\ell_N}^j(\eta),
\end{align}

\end{Lemma}
\begin{Rem}
    It will be convenient to note an equivalent expression for $\mathbf{h}_{\beta,\ell_N}$. Let $\ndef{\mathcal{P}_{n}}(A)$ be the collection of subsets of some set $A\subset\T_N$ with exactly $n$ elements. One can then express (see the last two displays in the proof of \cite[Proposition 2.7]{N24} for the precise computations) 
\begin{align}\label{h:alt}
    \mathbf{h}_{\beta,\ell_N}(\eta)
     &=
    \frac{1}{\ell_N+1}
    \sum_{n=0}^{\ell_N}
    \beta(\tfrac{n}{\ell_N})
    \sum_{i=n}^{\ell_N}
    \sum_{
    P\in\mathcal{P}_{i-n}(\llbracket0,i-1\rrbracket)
    }
    \prod_{p\in P}
    (1-\eta(p))
    \prod_{q\in \llbracket0,i\rrbracket\backslash P}
    \eta(q)
    .
\end{align}    
\end{Rem}
\subsection{Topological setting}\label{subsec:analysis}
We now present the analysis context of this work. Fixed a finite \textit{time-horizon} $\ndef{[0,T]}$, let $\ndef{\mu_N}$ be an initial probability measure on $\Omega_N$, and let $\{\eta_{N^2t}\}_{t\geq 0}$ be the process generated by $\mathfrak{L}_{N}$ in Definition \ref{def:model}. Our main object of study is the \textit{empirical measure}, $ \pi^N $, the random measure given by 
\begin{align*}
    \pi^N(\cdot,\mathrm{d}u)
    =\frac1N \sum_{x\in\T_N}
    \delta_{\tfrac{x}{N}}
    (\mathrm{d}u)\pi_x(\cdot)
    ,
\end{align*}
where $ \ndef{\delta_{v}} $ is the Dirac measure at $v\in\mathbb{T}$. For each $\eta\in\Omega_N$ fixed, its (diffusive) time evolution is defined as $ \ndef{\pi^N_t}(\eta,\rmd u)=\pi^N(\eta_{N^2t},\rmd u)$;  and for any function $ G:\mathbb{T}\to\mathbb{R} $ we shorten the integral of $ G $ with respect to the empirical measure as
\begin{align}\label{int:emp}
		\pi_t^N(\eta,G)
		=\int_{\mathbb{T}}G(u)\pi_t^N(\eta,\rmd u)
		.
\end{align}

Denote by $ \ndef{\mathcal{M}_+} $ the space of positive measures on $ [0,1] $ with total mass at most $ 1 $ and endowed with the weak topology. The Skorokhod space of trajectories induced by $ \{\eta_{N^2t}\}_{t\in[0,T]} $ with initial measure $ \mu_N $ is denoted by $ \ndef{\mathcal{D}([0,T],\Omega_N)} $, and we denote by $ \ndef{\mathbb{P}_{\mu_N}} $ the induced probability measure on it. Moreover, $ \ndef{\mathbb{Q}_N}:=\mathbb{P}_{\mu_N}\circ(\pi^N)^{-1} $ is the probability measure on $ \mathcal{D}([0,T],\mathcal{M}_+) $ induced by $ \{\pi^N_t\}_{t\in[0,T]} $ and $ \mu_N $.

For $ p\in\mathbb{N}_+\cup\{\infty\} $, let $ \ndef{C^p}(\mathbb{T}) $ be the set of $ p $ times continuously differentiable, real-valued functions defined on $ \mathbb{T} $; and let $ \ndef{C^{q,p}}(\mathbb{T}\times [0,T]) $ be the set of all real-valued functions defined on $ \mathbb{T}\times [0,T] $ that are $ q $ times differentiable on the space variables, in $\T$, and $p$ times differentiable on the time variable, in $[0,T]$, with continuous derivatives. For $f,g\in L^2(\mathbb{T})$, we denote by $\ndef{\langle f,g\rangle}$ their standard Euclidean product in $L^2(\mathbb{T})$ and $\ndef{\|\cdot\|_{2}}$ its induced norm.

We now aim to introduce the relevant weak formulation of the hydrodynamic equation. To that end, for any pair $ G,H\in C^\infty(\mathbb{T}^d) $ let $ \ndef{\inner{G,H}_1}=\inner{\partial_uG,\partial_u H}$ be their semi inner-product on $C^{\infty}(\mathbb{T})$, and $ \ndef{\norm{\cdot}_1} $ its associated semi-norm. The space $ \ndef{\mathcal{H}^1}(\mathbb{T}) $ is the Sobolev space on $ \mathbb{T} $, defined as the completion of $ C^\infty(\mathbb{T}) $ for the norm $ \ndef{\norm{\cdot}_{\mathcal{H}^1(\mathbb{T})}^2}:=\norm{\cdot}^2_{L^2}+\norm{\cdot}_{1}^2 $. We write as $ \ndef{L^2([0,T];\mathcal{H}^1(\mathbb{T}))} $ the set of measurable functions $ f:[0,T]\to\mathcal{H}^1(\mathbb{T}) $ such that $ \int_0^T\norm{f_s}^2_{\mathcal{H}^1(\mathbb{T})}\rmd s<\infty $.

The next definition corresponds to the notion of weak solution considered here.

\begin{Def}\label{def:weak}
    For $ \rho^{\rm ini}:\mathbb{T}^d\to[0,1] $ a measurable function, we say that $ \rho:[0,T]\times \mathbb{T}\mapsto [0,1] $ is a weak solution of the equation
        \begin{align}\label{hydro-eq}
		\begin{cases}
		    \partial_t\rho
            =\partial_u^2\Phi(\rho), & \text{in } (0,T]\times\T^d,\\
            \rho_0=\rho^{\text{ini}},& \text{in }\mathbb{T}^d
		\end{cases} 
	\end{align}
    if
	\begin{enumerate}
    \item $\Phi(\rho)\in L^2([0,T];\mathcal{H}^1(\mathbb{T}^d))$;
	\item for any $ t\in[0,T] $ and $ G\in C^{1,2}([0,T]\times \mathbb{T}) $, $\rho$ satisfies the formulation $\mathfrak{F}_t(\rho^{\rm ini},\rho,G)=0$, where   
		\begin{equation}\label{weak}
			\mathfrak{F}_t(\rho^{\rm ini},\rho,G)
			:=
			\inner{\rho_t,G_t}-\inner{\rho^{\rm ini},G_0}
			-\int_0^t
			\bigg\{
			\inner{\rho_s,\partial_sG_s}
			+
                \inner{\Phi(\rho_s),\partial_u^2 G_s}
			\bigg\}
			\rmd s
			.
		\end{equation}
	\end{enumerate}
\end{Def}

The uniqueness of solutions of the weak formulation in last definition is consequence of the regularity in (1). This is very straightforwardly shown via Oleinik's method (see, for instance, \cite[Lemma B.3]{s2f}).

\subsection{Main Result}\label{subsec:hl}
The hydrodynamic limit establishes a weak law of large numbers: starting from a \textit{local equilibrium distribution}, the empirical measure converges weakly to an absolutely continuous measure, whose density is the unique solution of the hydrodynamic equation \eqref{hydro-eq}. Note that, indirectly, this establishes the existence of solutions of \eqref{hydro-eq} in the weak sense of Definition \ref{def:weak}. We aim to present this statement rigorously.
\begin{Def}[Local equilibrium distribution]\label{def:ass}
	Let $ \{\mu_N\}_{N\geq 1} $ be a sequence of probability measures on $ \Omega_N $, and let $g:\mathbb{T}\to[0,1] $ be a measurable function. If, for any continuous function $ G:\mathbb{T}\to\mathbb{R} $ and every $ \delta>0 $, it holds
	\begin{align*}
		\lim_{N\to+\infty}\mu_N
		\left(
		\eta\in\Omega_N
		:\big|\pi^N(\eta,G)-\inner{g,G}\big|>\delta
		\right)
		=0,
	\end{align*}
	we say that the sequence $ \{\mu_N\}_{N\geq 1} $ is a local equilibrium measure associated with the profile $ g $.
\end{Def}
A simple example of local equilibrium measures are the Bernoulli product measures with, for example, a continuous parameter-function. We now state the main theorem.
\begin{Th}[Hydrodynamic limit] \label{th:hydro}
	Let $ \rho^{\rm ini}:\mathbb{T}\to[0,1] $ be a measurable function and let $ \{\mu_N\}_{N\geq 1} $ be a local equilibrium measure associated with it.
	Then, for any $ t\in [0,T] $ and $ \delta>0 $, it holds
	\begin{align*}
		\lim_{N\to+\infty}
		\mathbb{P}_{\mu_N}
		\left(
		\big|\pi_t^N(\eta,G)-\inner{\rho_t,G}\big|>\delta
		\right)=0,
	\end{align*}
	where $ \rho $ is the unique solution of the diffusion equation \eqref{hydro-eq}, in the sense of Definition \ref{def:weak}, with initial data $\rho^{\rm ini}$ and 
    \begin{align}
        \Phi(\rho)=\int_0^\rho\beta(u)\rmd u
        .
    \end{align}
\end{Th}

\section{Proof of the Hydrodynamic Limit}\label{sec:hl}
In order to prove Theorem \ref{th:hydro}, we apply the entropy method, first introduced in \cite{entropy}. We start by presenting the general scheme of the proof, where we discuss the tightness of the measure and regularity of the density. Next, we proceed with the characterization of the limit measure of the process.

\subsection{General scheme}\label{subsec:scheme}
The link between our process and the weak formulation \eqref{weak} is given through Dynkin's martingale (see \cite[Appendix 1, Lemma 5.1]{KL:book}),
\begin{align}\label{eq:M}
	\text{M}_N^G(t)
	:=\pi_t^N(\eta,G_t)-\pi_0^N(\eta,G_0)
	&-\int_0^t
	\pi_s^N(\eta,\partial_sG_s)+\mathfrak{L}_N[\pi_s^N](\eta,G_s)\rmd s
\end{align}
for $G\in C^{2,1}(\mathbb{T}\times [0,T])$.

One can show that the sequence of probability measures $ (\mathbb{Q}_N)_{N\in\mathbb{N}} $ is tight with respect to the Skorokhod topology of $ \mathcal{D}\left([0,T],\mathcal{M}_+\right) $ by following \cite[Subsection 3.2]{N25}, which invokes Aldous' conditions. As evidenced in \cite[Proposition 3.1]{s2f} and \cite[Subsection 3.2]{N25}, tightness for exclusion processes in the torus satisfying the gradient condition is consequence of the following: \textbf{\textit{(i)}} $|\mathbf{r}_{\beta,\ell_N}|_{\infty}=o(N)$ ; \textbf{\textit{(ii)}} $|\mathbf{h}_{\beta,\ell_N}|_{\infty}$ being uniformly bounded from above by a constant independent of $N$; \textbf{\textit{(iii)}} and the estimate in Lemma \ref{lem:g-fun} vanishing as $N\to\infty$. The former two are straightforwardly identified, therefore we refer the reader to the aforementioned references for the specific details to include them in the proof.  

With this, we conclude that the sequence of empirical measures is tight, existing then weakly convergent subsequences. Since there is at most one particle per site, one can show that the limiting measure of a convergent subsequence of $(\mathbb{Q}_N)_{N\geq 0}$, that we write as \ndef{$\mathbb{Q}$}, is concentrated on paths of absolutely continuous measures with respect to the Lebesgue measure. This means precisely that the sequence $ (\pi_\cdot^{N}(\eta,\rmd u))_{N\in\mathbb{N}} $ converges weakly, with respect to $ \mathbb{Q}_N $, to an absolutely continuous measure with a density that we write as \ndef{$\rho$}, that is, $ \pi_{\cdot}(\rmd u)=\rho_\cdot(u)\rmd u $. 

Provided the aforementioned convergence by subsequences and absolutely continuity of the limit measure, one then argues that the density $\rho$ satisfies the notion of weak solution in Definition \ref{def:weak}. This encompasses two results to be held $\mathbb{Q}$ almost-surely: \textit{\textbf{(i)}} satisfaction of the weak formulation $\mathfrak{F}_t(\rho^{\text{ini}},\rho,G)=0$; and \textit{\textbf{(ii)}} regularity, as in (1) in Definition \ref{def:weak}. The latter is typically shown via the \textit{so-called} "energy estimate". This is standard in the literature, and shown to hold for models with properties satisfied by the model in the present text. Concretely, proved the "replacement lemmas" (here stated in Corollary \ref{lem:1block} and \ref{lem:2block}), and the vanishing, as $N\to\infty$, of the remainder term in the gradient property, \eqref{g}, as in Lemma \ref{lem:g-fun}, the regularity property (1) is consequence of $\mathbf{c}_{\beta,\ell_N}$ being uniformly bounded from above by a constant independent of $N$. For the precise computations, we refer the reader to \cite[Section 5]{s2f}, where this was shown for the "Interpolating process" in the slow-diffusion regime; or \cite[Section 5 -- Regime I]{N25}. The latter is closer, in terms of notation and exposition, to the present text. The next step is to characterize the limit $\mathbb{Q}$.

\subsection{Characterization of the limiting points}
We now focus on showing that for any limit point $ \mathbb{Q} $ of a convergent subsequence of $(\mathbb{Q}_N)_{N\in\mathbb{N}} $ it holds 
\begin{align}\label{eq:char}
        \mathbb{Q}
	\big(\pi_{\cdot} 
        :\;
        \mathfrak{F}_t(\rho^{\rm ini},\rho,G) =0 
        ,\;
            \forall t\in[0,T],\;
        \forall G\in C^{1,2}([0,T]\times\mathbb T)
        \;\big|\;
        \tfrac{\pi(\rmd u)}{\rmd u}=\rho
        \big)
	=1
    ,
\end{align}
where we recall that $\mathfrak{F}_t(\rho^{\rm ini},\rho,G)$ is given in  \eqref{weak}. The next steps up to \eqref{h-to-rep} are presented in detail in \cite[Proof of Proposition 3.2]{N25}, from the start of the proof up to equation (3.12) in the aforementioned work. Here, we only overview the approach and present the main novelties.

At this point is necessary to introduce a discretization of $\Phi$. Fix $\alpha\in(0,1)$ and denote by $\ndef{\nu_\alpha^N}$ the Bernoulli product measure parametrised by $\alpha$, that is, defined by the marginals $\nu_\alpha^N(\eta\in\Omega_N\;:\;\eta(x)=1)=\alpha$. Introduce the function $\ndef{\Phi_{\beta,L}}:[0,1]\to\mathbb{R}_+\cup\{0\}$, for each $L\in\mathbb{N}_+$ fixed, through
\begin{align}\label{phiL}
	\Phi_{\beta,L}(\alpha)
	:=\text{E}_{\nu_{\alpha}^{N}}[\mathbf{h}_{L}]
	,
\end{align}
Explicitly,
\begin{align}
    \Phi_{\beta,L}(u)
    =\sum_{n=0}^L\beta(\tfrac{n}{L})
    \text{H}_{n,L}(u)
    \quad\text{with}\quad
    \text{H}_{n,L}(u)
    =\frac{1}{L+1}
    \sum_{i=n}^{L}
    \binom{i}{n}u^{n+1}(1-u)^{i-n}
    .
\end{align}
From Weierstrass' approximation theorem, recalling \eqref{binom} and noticing that $\Phi_L'=\text{B}_{\beta,L}$, the sequence $(\Phi_{\beta,L})_{L\geq 0}$ converges uniformly to 
\begin{align}
        \Phi_\beta(\rho)=\int_0^\rho\beta(u)\rmd u
        .
\end{align}
Moreover, from Lemma \ref{lem:grad} and the expression in \eqref{g}, $\text{E}_{\nu_{\alpha}^{N}}[\mathbf{H}_{\beta,\ell_N}] = \text{E}_{\nu_{\alpha}^{N}}[\mathbf{h}_{\beta,\ell_N}]$.

Let us now introduce the functional $\ndef{\tilde{\Phi}_{\beta,L}}$, corresponding $\tilde{\Phi}_{\beta,L}[f]$, for each $f:\T_N\to\mathbb{R}$, to the expression resulting from replacing $\eta(x)$ by $f(x)$, for each $x\in\T_N$ in the expression of $\mathbf{h}_{\beta,L}(\eta)$ in \eqref{h}. The object $\tilde{\Phi}_{\beta,L}$ is a natural linearization of $\Phi_{\beta,L}$ provided by the dynamics, in the sense that $\Tilde{\Phi}_{\beta,L}[\rho_s((\cdot)/N)]$ is linear on any element of $\{\rho(x/N)\}_{x\in\T_N}$. 

Fixed $ \eps>0 $ and $ u\in\T $, define the cut-off function \ndef{$\iota_\eps^{u}$} through
\begin{align}
    {\iota}_\eps^{u}(v)
    :=
    \frac{1}{\eps}\mathbf{1}_{B_{\eps}(u)}(v), 
\end{align}
for each $v\in\T$, where $\ndef{B_\eps}(u)=[u,u+\eps)$, and where $\ndef{\mathbf{1}_A}(v)=1$ if $v\in A$ and zero otherwise. Shortening $\ndef{\Psi_u^{\eps,L}}[\pi_s]\equiv\Tilde{\Phi}_{\beta,L}[\pi_s(\iota_\eps^{u+(\cdot)\eps})]$, from Lebesgue's differentiation theorem one can show that 
\begin{align}
    \lim_{L\to+\infty}\lim_{\eps\to0}
    \big|\Phi_L(\rho_s(u))
    -\Psi_u^{\eps,L}[\pi_s]\big|
    =0,
\end{align}
reducing the proof of \eqref{eq:char} to showing that, for arbitrary $\delta>0$,
\begin{align}
    \limsup_{L\to\infty}
    \lim_{\eps\to0}
\mathbb{Q}
\bigg(\pi_\cdot:\;
    \sup_{t\in[0,T]}
    \bigg|
    \pi_t(G_t)-\pi_0(G_0)
    -\int_0^t
    \pi_s(\partial_sG_s)
    \rmd s
\\-
    \int_0^t
    \int_{\T}
    \partial_{u}^2 G_s(u)
    \Tilde{\Phi}_{\beta,L}[\pi_s(\iota_\eps^{u+(\cdot)\eps})]
    \rmd u\rmd s
    \bigg|>\delta
\bigg)=0.
\end{align}

It is now convenient to introduce the discrete Laplacian operator. For each $F:\T\to\mathbb{R}$, let $\ndef{\Delta_N}$ be defined through $\Delta_NF(u)=N^2[F(u-\tfrac1N)-2F(u)+F(u+\tfrac1N)]$, for any $u\in \T$. It is also important to see that from the gradient property, that is, Lemma \ref{lem:grad}, it holds 
\begin{align}
    \mathfrak{L}_N[\pi_x]=N^2\Delta\tau^x\mathbf{H}_{\beta,\ell_N}.
\end{align}

With the above in mind, from Portmanteau's theorem and then several times the Markov and triangle's inequality, we are reduced to the analysis of the limits $\limsup_{L\to+\infty}\lim_{\eps\to0}\limsup_{N\to+\infty}$ of each of the following terms
\begin{align}
&\mathbb{P}_{\mu_N}
\bigg(
\sup_{t\in[0,T]}
\bigg|
\text{M}_N^G(t)
\bigg|
>\delta
\bigg)
,
\\&
\mathbb{P}_{\mu_N}
\bigg(
\sup_{t\in[0,T]}
\bigg|
    \int_0^t
        \frac{1}{N}
        \sum_{x\in\T_N}
		\Delta_N G_s(\tfrac{x}{N})
            \mathbf{g}_{\beta,\ell_N}
            (\tau^x\eta_{N^2 s})
    \rmd s
\bigg|>\delta
\bigg)
,\\\label{h-cont}
&
\mathbb{P}_{\mu_N}
\bigg(
\sup_{t\in[0,T]}
\bigg|
\int_0^t
    \frac1N 
    \sum_{x\in\T_N}
    \Delta_NG(s,\tfrac{x}{N})
\bigg\{
\mathbf{h}_{\beta,\ell_N}(\tau^x\eta_{N^2s})
-
\mathbf{h}_{\beta,L}(\tau^x\eta_{N^2s})
\bigg\}
\rmd s
\bigg|
>\delta
\bigg)
,\\\label{h-to-rep}
&
\mathbb{P}_{\mu_N}
\bigg(
\sup_{t\in[0,T]}
\bigg|
    \int_0^t
        \frac{1}{N} 
        \sum_{x\in\T_N}
            \Delta_N G_s(\tfrac{x}{N})
            \bigg\{
            \Tilde{\Phi}_{\beta,L}[\pi_s(\eta,\iota_\eps^{\tfrac{x}{N}+(\cdot)\eps})]
                -
                \mathbf{h}_{\beta,L}(\tau^x\eta_{N^2s})
            \bigg\}
    \rmd s
\bigg|
>\delta
\bigg)
,\\
&
\mathbb{P}_{\mu_N}
\bigg(
\sup_{t\in[0,T]}
\bigg|
\int_0^t
\int_{\T}
    \partial_u^2 G_s(u)
    \Tilde{\Phi}_{\beta,L}[\pi_s(\eta,\iota_\eps^{u+(\cdot)\eps})]   
    \rmd u
\rmd s
\\&
\qquad\qquad\qquad
-\int_0^t
    \frac1N
    \sum_{x\in\T_N}
    \Delta_NG_s(\tfrac{x}{N})
    \Tilde{\Phi}_{\beta,L}[\pi_s^N(\eta,\iota_\eps^{\tfrac{x}{N}+(\cdot)\eps})]
\rmd s
\bigg|
>\delta
\bigg).
\end{align}
We analyse each of the terms above:
\begin{itemize}
    \item After an application of Doob's inequality, one can see that the first term vanishes in the limit $N\to+\infty$ by repeating the computations to show the first equality in \cite[Equation (3.2)]{N25}, where $|\mathbf{r}_{\beta,\ell_N}|_\infty=o(N)$ is required;

    \item For the second, one applies \cite[Proposition 4.3]{N25} and then Lemma \ref{lem:g-fun};

    \item The last term is analysed by performing a Taylor expansion, then approximating the integral by a Riemann sum. 
\end{itemize}
In order to conclude the proof, we are left with analysing \eqref{h-cont} and \eqref{h-to-rep}, which is where novel arguments must be introduced.

Let us focus on \eqref{h-cont}, treated in the forthcoming Lemma \ref{lem:h-treat} and recall $\mathbf{h}_{\beta,\ell_N}$ as in \eqref{h}. From here on, for each $l\in\mathbb{N}_+$ consider the ball $\ndef{B_{l}}=\llbracket 0,l-1\rrbracket$, and shorten for any $\eta\in\Omega_N$ the density $\ndef{\inner{\eta}_l}$ as $\inner{\eta}_{l}:=l^{-1}\sum_{y\in B_{l}}\eta(y)$.   
\begin{Lemma}\label{lem:h-treat}
For each fixed integer $L>0$ it holds that 
    \begin{multline}
        \limsup_{L\to+\infty}\limsup_{\eps\to0}\limsup_{N\to+\infty}        \mathbb{P}_{\mu_N}
        \bigg(
        \sup_{t\in[0,T]}
        \bigg|
        \int_0^t
        \frac1N 
        \sum_{x\in\T_N}
        \Delta_NG_s(\tfrac{x}{N})
        \times\\\times\bigg\{
        \mathbf{h}_{\beta,\ell_N}(\tau^x\eta_{N^2s})
        -
        \mathbf{h}_{\beta,L}(\tau^x\eta_{N^2s})
        \bigg\}
        \rmd s
        \bigg|
        >\delta
        \bigg)
        =0
        .
    \end{multline}
\end{Lemma}
\begin{proof}
    
For each $M\in\mathbb{N}_+$, from Weierstrass' approximation theorem,
\begin{align}\label{eq:reg-h}
    \mathbf{h}_{\beta,\ell_N}(\eta)
    -\mathbf{h}_{\beta,L}(\eta)
    &=
    \sum_{m=0}^M\beta(\tfrac{m}{M})
\big\{
    \frac{1}{\ell_N+1}
    \sum_{n=0}^{P_{\ell_N}(\eta)-1}
    \text{B}_{m,M}(\tfrac{n}{\ell_N})
    -
    \frac{1}{L+1}
    \sum_{n=0}^{P_L(\eta)-1}
    \text{B}_{m,M}(\tfrac{n}{L})
\big\}
    \\&+\text{e}_0(M)
\end{align}
where $\lim_{M\to\infty}\lim_{L\to\infty}\lim_{N\to\infty}|\text{e}_0(M)|=0$. Let us introduce 
\begin{align}\label{prim-B}
    \text{H}_{m,M}(v):=\int_0^v\text{B}_{m,M}(u)\rmd u, \qquad\text{for every } v\in(0,1].
\end{align}
We identify the summations over $n$ in \eqref{eq:reg-h} as Riemann sums, and because the map $u\mapsto \text{B}_{m,M}(u)$ is continuous and differentiable, the quantity in \eqref{eq:reg-h} equals almost surely
\begin{multline}
    \sum_{m=0}^M\beta(\tfrac{m}{M})
\bigg\{
    \int_{0}^{\inner{\eta}_{\ell_{N}}}
    \text{B}_{m,M}(u)\rmd u
    -
    \int_{0}^{\inner{\eta}_L}
    \text{B}_{m,M}(u)\rmd u
\bigg\}
    +\text{e}_0(M)
    +\text{e}_1(\ell_N)
    +\text{e}_1(L)
    \\
    =\sum_{m=0}^M\beta(\tfrac{m}{M})
\bigg\{
    \text{H}_{m,M}(\inner{\eta}_{\ell_N})
    -\text{H}_{m,M}(\inner{\eta}_{L})
\bigg\}
    +\text{e}(M,\ell_N,L)
\end{multline}
where $|\text{e}_1(\ell)|\xrightarrow{\ell\to\infty}0$ and we shortened $\text{e}(M,\ell_N,L)\equiv \text{e}_0(M)+\text{e}_1(\ell_N)+\text{e}_1(L)$.

In this way, it is enough to show that for any $\delta'>0$,
\begin{multline}
\lim_{M\to\infty}\lim_{L\to\infty}\limsup_{N\to\infty}
\mathbb{P}_{\mu_N}
    \bigg(
    \eta_{\cdot}: \;
    \sup_{t\in[0,T]}
    \bigg |
    \int_{0}^t
    \frac1N \sum_{x\in\mathbb{T}_N}G_s(\tfrac{x}{N})
    \sum_{m=0}^M\beta(\tfrac{m}{M})
\times\\\times
    \bigg\{
    \text{H}_{m,M}(\inner{\tau^x\eta_{N^2s}}_{\ell_N})
    -\text{H}_{m,M}(\inner{\tau^x\eta_{N^2s}}_{L})
\bigg\}
    \rmd s
    \bigg|
    >\delta'
    \bigg)
    =0
.
\end{multline}
Applying \cite[Proposition 4.3]{N25} and the triangle inequality, the limit above is a direct consequence of forthcoming Lemma \ref{lem:rep-h}.

\end{proof}

We now focus on \eqref{h-to-rep}. Having in mind the expression for $\mathbf{H}_{\beta,\ell_N}$ as in \eqref{h:alt}, and that for any sequence of real numbers $(a_i)_i,(b_i)_i$ one can rearrange
\begin{align}\label{reorg}
    \prod_{i=0}^La_i
    -\prod_{i=0}^Lb_i
    =
    \sum_{m=0}^{L}(a_m-b_m)\prod_{\substack{i=0\\i\neq m}}^{L}c_i
    \where
    c_i=\begin{cases}
        a_i, &i< m,\\b_i, &i>m,
    \end{cases}
\end{align} 
then for any set $A\subset\T_N$ and $P\subset A$ one can express
\begin{align}
    &\prod_{p\in P}
    \pi_s^N(\eta,\iota_{\eps}^{p\eps})
    \prod_{q\in A\backslash P}
    (1-\pi_s^N(\eta,\iota_{\eps}^{q\eps}))
    -
    \prod_{p\in P}\eta(p)
    \prod_{q\in A\backslash P}(1-\eta(q))
    \\&=
    \sum_{m\in A}
    \varphi_m^{\eps,\eps N}(\eta)
\bigg\{
    \pi_s^N(\eta,\iota_{\eps}^{m\eps})
-
    \inner{\tau^{\ceil{m\eps N}}\eta}_{\ceil{\eps N}}
\bigg\}
\label{embbeb}\\
 &+    \sum_{m\in A}
  \varphi_m^{\eps N,L}(\eta)
 \bigg\{
     \inner{\tau^{\ceil{m\eps N}}\eta}_{\ceil{\eps N}}
 - 
     \inner{\tau^{mL}\eta}_{{L}}
 \bigg\}
 \label{two-block}\\
&
+    \sum_{m\in A}
 \varphi_m^{L}(\eta)
\bigg\{
\inner{\tau^{mL}\eta}_{{L}}
 -\eta(mL)
\bigg\} 
\label{one-block}\\&
 +
     \sum_{m\in A}
    \varphi_m^{L}(\eta)
\bigg\{
\eta(mL)
 -
 \eta(m)
 \bigg\}\label{mix},
\end{align}
where, for each $m\in A$, 
\begin{itemize}
    \item $\varphi_m^{\eps N,L}$ is independent of the occupation value at $\llbracket\ceil{m\eps N},\ceil{m\eps N}+\ceil{\eps N}\rrbracket\cup\llbracket mL,mL+L\rrbracket$;

    \item $\varphi_m^{L}$ is independent of the occupation value at $\llbracket mL,mL+L\rrbracket\cup\llbracket mL,mL+L\rrbracket$;

    \item $\varphi_m^{L}$ is independent of the occupation value at $\{mL,m\}$;

    \item Each of the previous maps and $\varphi_m^{\eps,\eps N}$ are uniformly bounded from above by $1$.

\end{itemize}
We aim to show that that the $\lim_{L\to\infty}\lim_{\eps\to0}\limsup_{N\to\infty}$ of the probability in \eqref{h-to-rep} is zero, and for that we split the aforementioned probability into four probabilities, each associated with a translation by $+x$ of each term of the decomposition in the previous display. 

For \eqref{embbeb}, from Markov's inequality it is enough to see that for each $0\leq m\leq L$,
\begin{align}
\Big|
\pi_s^N(\eta,\iota_{\eps}^{m\eps})
-
\inner{\tau^{\ceil{m\eps N}}\eta}_{\ceil{\eps N}}
\Big|
=
\Big|    \frac{1}{\eps N}
    \sum_{y\in B_{\eps N}(m\eps N)\cap \T_N}\eta(y)
    -\frac{1}{\ceil{\eps N}}
    \sum_{y\in B_{\ceil{\eps N}}(\ceil{m\eps N})}\eta(y)
\Big|
\end{align}
and that the quantity in the right-hand-side above vanishes in the limits $\lim_{\eps\to 0}\limsup_{N\to\infty}$.

Next, we apply \cite[Proposition 4.3]{N25} to each of the remaining probabilities and analyse them separately with the Replacement Lemmas in Section \ref{sec:rep-lem} (Corollaries \ref{lem:1block} and \ref{lem:2block}), after an application of the triangle inequality in order to pass the summations over $m$ to outside of the expectation. Precisely, the terms associated with \eqref{mix} and \eqref{one-block} are analysed with the upcoming Lemma \ref{lem:1block}, with $L$ in the statement of the aforementioned lemma either equal to $1$ or an arbitrary integer $L>1$, respectively, and $w$ and $y$ chosen accordingly. To treat the term associated with \eqref{two-block}, it is enough to apply the triangle inequality and then invoke Lemma \ref{lem:2block}. In order to conclude the proof, we perform the sequence of limits $\limsup_{M\to+\infty}\limsup_{L\to+\infty}\limsup_{\eps\to0}\limsup_{N\to+\infty}$.

\section{Replacement Lemmas}\label{sec:rep-lem}
The goal of this section is to prove the Lemmas \ref{lem:g-fun} and \ref{lem:rep-h}. In order to do so, we introduce the next objects. We recall the generator of the process as in Definition \ref{def:model}. Fixed any constant $\alpha\in(0,1)$, let $g:\Omega_N\to\mathbb{R}$ and introduce the average Dirichlet form $\ndef{\mathfrak{D}_N}$ through
    \begin{align}\label{revers-dir}
        \frac12
        \mathfrak{D}_N(g)
        =
        \int_{\Omega_N}
        g
        (-\mathfrak{L}_N)g
        \;\rmd\nu_\alpha^N
        ,
    \end{align}
We say that $f:\Omega_N\to\mathbb{R}^+\cup\{0\}$ is a density (with respect to $\nu_\alpha^N$) if $\nu_\alpha^N(f)=1$. 

It will be convenient to note that, fixed $x,y\in\T_N$, for any $\varphi:\Omega_N\to\mathbb{R}$ independent of the transformation $\eta\mapsto\theta_{x,y}\eta$ and any $g:\Omega_N\to\mathbb{R}$, it holds that
        \begin{align}\label{by-parts}
            \int_{\Omega_N}\varphi(\eta)
            (\eta(x)-\eta(y))
            g(\eta)\rmd\nu_\alpha^N
            =-\frac12\int_{\Omega_N}\varphi(\eta)
            (\eta(x)-\eta(y))
            \nabla_{x,y}
            g(\eta)\rmd\nu_\alpha^N
            .
        \end{align}

For $\mu$ and $\nu$ two probability measures on $\Omega_N$, the relative entropy of $\mu$ with respect to $\nu$, that we write as $\ndef{\text{H}(\mu|\nu)}$, is defined as
    \begin{align}
        \text{H}(\mu|\nu)
        :=\sup_{g:\Omega_N\to\mathbb{R}}
        \big\{
            \mu(g)
        -\log\nu(e^g)    
        \big\}.
    \end{align}
One can show that there exist a positive constant $\ndef{\text{c}_{\alpha}}$ dependent of $\alpha$ and independent of $N$, such that (see \cite[proof of Proposition 4.2]{N25})
\begin{align}\label{c-alpha}
    \text{H}(\mu_N|\nu_\alpha^N)\leq \text{c}_{\alpha} N.
\end{align}

We now state and prove the required replacements.

\begin{Lemma}\label{lem:g-fun}
For each $x\in\T_N$ fixed and every $t\in [0,T]$, let $\varphi:[0,T]\times \T_N\to\mathbb{R}$ be such that $M_{t,\varphi}\equiv\int_0^t\norm{\varphi(s,\cdot)}_{\infty}\rmd s<\infty$. It holds that
    \begin{align}
        \mathbb{E}_{\mu_N}
        \bigg[
        \bigg|
        \int_{0}^t
        \frac1N \sum_{x\in\T_N}\varphi(s,x)
        \mathbf{g}_{\beta,\ell_N}(\tau^x\eta_{N^2s})
        \rmd s
        \bigg|
        \bigg]
        \leq 
        \sqrt{2 c_{\alpha}}M_{T,\varphi}
        \frac{\ell_N}{\sqrt{TN^2/|\mathbf{r}_{\beta,\ell_N}|_{\infty}}}
        ,
    \end{align}
with $c_\alpha$ as in \eqref{c-alpha} and $\mathbf{g}_{\beta,\ell_N}$ as in \eqref{g}.

\end{Lemma}
\begin{proof}
From a standard procedure involving the entropy inequality and Feynmann-Kac's formula (see, for example, \cite[Proposition 4.2]{N25}), the expectation on the left-hand side of the previous display can be bounded from above by the following quantity, where $B$ is an arbitrary positive constant, that shall be fixed later on:
    \begin{align}\label{rest:var}
        \frac{c_{\alpha}}{B}
        +
    \int_0^t
        \sup_{f \text{ density}}
        \bigg\{
        \frac1N \sum_{x\in\T_N}\varphi(s,x)
        \int_{\Omega_N}
        \tau^x\mathbf{g}_{\beta,\ell_N}
        f\rmd\nu_\alpha^N
        -\frac{\mathfrak{D}_{N}(\sqrt{f}|\nu_\alpha^N)}{2N B}
        \bigg\}
    \rmd s
        .
    \end{align}
From \eqref{by-parts} and Young's inequality, and recalling $\mathbf{r}_{\beta,\ell_N}^j$ as in \eqref{rate-j},
\begin{align}
&\bigg|
    \frac1N \sum_{x\in\T_N}\varphi(s,x)
    \int_{\Omega_N}
    \tau^x\mathbf{g}_{\beta,\ell_N}
    f\rmd\nu_\alpha^N
\bigg|
\\&\label{rest:to-dir}
\leq
\frac{\norm{\varphi(s,\cdot)}_{\infty}}{4NA}
\int_{\Omega_N}
\sum_{x\in\T_N}
\frac{1}{\ell_N+1}
    \sum_{j=1}^{\ell_N}
    \sum_{z=1}^{j}
    \tau^{-z+x}\mathbf{r}_{\beta,\ell_N}^j    
    \big|
    \nabla_{x+z+1,x+z}[\sqrt{f}]
    \big|^2
\rmd\nu_\alpha^N
\\&\label{rest:to-zero}
+
\frac{A\norm{\varphi(s,\cdot)}_{\infty}}{2}
\int_{\Omega_N}
\frac1N\sum_{x\in\T_N}
\frac{1}{\ell_N+1}
    \sum_{j=1}^{\ell_N}
    \sum_{z=1}^{j}
    \tau^{-z+x}\mathbf{r}_{\beta,\ell_N}^j
    \big(
    \theta_{x+z+1,x+z}+\mathbf{1}
    \big)[f]
\rmd\nu_\alpha^N
,
\end{align}
for an arbitrary $A>0$. Let us focus on the term \eqref{rest:to-zero}. Performing the change of variables $\eta\mapsto\eta^{x+z,x+z+1}$ and recalling that $\alpha$ is a constant profile, the integral in \eqref{rest:to-zero} is bounded from above by
\begin{multline}
    \int_{\Omega_N}
    f
    \frac1N\sum_{x\in\T_N}
    \frac{1}{\ell_N+1}
    \sum_{j=1}^{\ell_N}
    \sum_{z=1}^{j}
    \theta_{x+z+1,x+z}\tau^{-z+x}\mathbf{r}_{\beta,\ell_N}^j
\rmd\nu_\alpha^N
\\
+\int_{\Omega_N}
    f
    \frac1N\sum_{x\in\T_N}
    \frac{1}{\ell_N+1}
    \sum_{j=1}^{\ell_N}
    \sum_{z=1}^{j}
    \tau^{-z+x}\mathbf{r}_{\beta,\ell_N}^j
\rmd\nu_\alpha^N
\leq 
    2\ell_N
    |\mathbf{r}_{\beta,\ell_N}|_{\infty}
    ,
\end{multline}
from where we conclude that \eqref{rest:to-zero} is bounded from above by 
\begin{align}
    \norm{\varphi(s,\cdot)}_{\infty}
    |\mathbf{r}_{\beta,\ell_N}|_{\infty}
    \ell_N
    A
    .
\end{align}

We now focus on \eqref{rest:to-dir}. Exchanging the summations,
\begin{align}
\frac{\norm{\varphi(s,\cdot)}_{\infty}}{4NA}
\sum_{z=1}^{\ell_N}
    \int_{\Omega_N}
        \sum_{x\in\T_N}
        \frac{1}{\ell_N+1}
        \sum_{j=z}^{\ell_N}
        \tau^{-z+x}\mathbf{r}_{\beta,\ell_N}^j    
        \big|
        \nabla_{x+z+1,x+z}[\sqrt{f}]
        \big|^2
    \rmd\nu_\alpha^N
\\\leq
    \frac{\norm{\varphi(s,\cdot)}_{\infty}}{4NA}
    \frac{\ell_N\mathfrak{D}_N(\sqrt{f})}{N^2}
    ,
\end{align}
and \eqref{rest:var} can be bounded from above by
\begin{align}
\frac{c_{\alpha}}{B}
+
    M_{T,\varphi}
    |\mathbf{r}_{\beta,\ell_N}|_{\infty}
    \ell_NA
+   \frac{1}{2N}
    \mathfrak{D}_{N}(\sqrt{f})
    \bigg(
    \frac{M_{T,\varphi}\ell_N}{2AN^2}
    -\frac{T}{B}
    \bigg)
,
\end{align}
with $M_{T,\varphi}=\int_0^T\norm{\varphi(s,\cdot)}_{\infty}\rmd s$. Solving  
\begin{align}
    \begin{cases}
    \frac{c_{\alpha}}{B}
    =
    M_{T,\varphi}
    |\mathbf{r}_{\beta,\ell_N}|_{\infty}
    \ell_NA    
    \\
    \frac{M_{T,\varphi}\ell_N}{2AN^2}
    =\frac{T}{B}
    \end{cases}
\end{align}
for $A$ and $B$ yields the upper bound in the statement of this lemma.

\end{proof}

\begin{Lemma}\label{lem:rep-h}
Fixed $M\in\mathbb{N}_+$, for each $0\leq m\leq M$ and $u,v\in[0,1]$ let 
\begin{align}
    \mathfrak{V}_{m,M}(u,v)
    =
    \text{H}_{m,M}(u)
    -\text{H}_{m,M}(v)
\end{align}
with $\text{H}_{m,M}$ as in \eqref{prim-B}. For any $t\in[0,T]$, any $G:[0,T]\times \T\to\mathbb{R}$ such that  $c_{T,G}:=\int_{0}^T\norm{G_s}_{\infty}\rmd s<\infty$ and positive integer $L<\ell_N$ , it holds that 
\begin{align}
    \mathbb{E}_{\mu_N}
    \bigg[
    \bigg |
    \int_{0}^t
    \frac1N \sum_{x\in\mathbb{T}_N}G_s(\tfrac{x}{N})
\mathfrak{V}_{m,M}(\inner{\tau^x\eta_{N^2 s}}_{\ell_N},\inner{\tau^x\eta_{N^2 s}}_{L})
    \rmd s
    \bigg|
    \bigg]
    &\leq 
    \frac{c}{L}
    +
    c'
    \frac{\ell_N}{\sqrt{\min\{\mathbf{r}_{\beta,\ell_N}\}N^2T}}
    ,
\end{align}
with $c\equiv c_{T,G} c_{m,M}$ , $c'\equiv c_{m,M}'c_{T,G}\sqrt{c_\alpha}$ and where $c_{m,M},c_{m,M}'>0$ are constants dependent only on $m$ and $M$.

\end{Lemma}

\begin{proof}
Proceeding as in the previous lemma, it is enough to control
    \begin{align}\label{reg:var}
        \frac{c_{\alpha}}{B}
        +
    \int_0^t
        \sup_{f \text{ density}}
        \bigg\{
    \bigg|
        \frac1N \sum_{x\in\T_N}
        G_s(\tfrac{x}{N})
        \int_{\Omega_N}
\mathfrak{V}_{m,M}(\inner{\tau^x\eta}_{\ell_N},\inner{\tau^x\eta}_{L})
        f(\eta)\rmd\nu_\alpha^N(\eta)
    \bigg|
        -\frac{\mathfrak{D}_{N}(\sqrt{f})}{2N B}
        \bigg\}
    \rmd s
    ,
    \end{align}    
where $B>0$ will be fixed by the end of the proof. Noting that $\text{H}_{m,M}$ is an increasing and non-negative polynomial of degree $M+1$, we can express 
\begin{align}
\mathfrak{V}_{m,M}(\inner{\tau^x\eta}_{\ell_N},\inner{\tau^x\eta}_{L})
=
    \big(
    \inner{\eta}_{\ell_N}-\inner{\eta}_{L}
    \big)
    \text{V}_{m,M}(\inner{\eta}_{\ell_N},\inner{\eta}_{L})
\end{align}
for some \textit{non-negative} form $\text{V}_{m,M}\in C^{M,M}(\mathbb{T}\times\mathbb{T})$. It is important to identify that from Taylor's theorem
\begin{align}\label{V-reg}
    \big|
    \text{V}_{m,M}(\inner{\eta}_{\ell_N},\inner{\eta}_{L})
    -\text{V}_{m,M}(\inner{\eta}_{\ell_N},\inner{\eta}_{L}\pm\tfrac{1}{L+1})
    \big|
    \leq 
    \frac{1}{L+1}
    \norm{\partial_2
    \text{V}_{m,M}}_{\infty}
    .
\end{align}
Moreover, we also express
\begin{align}
    \inner{\eta}_{\ell_N}-\inner{\eta}_{L}
    =\frac{1}{(\ell_N+1)(L+1)}\sum_{w\in B_{\ell_N}}\sum_{z\in B_{L}}
    \big(
    \eta(w)-\eta(z)
    \big)
    .
\end{align}

With the previous information, we express further
\begin{align}
    &2\int_{\Omega_N}
    \mathfrak{V}_{m,M}(\inner{\tau^x\eta}_{\ell_N},\inner{\tau^x\eta}_{L})
    f(\eta)
    \rmd\nu_\alpha^N(\eta)
    \\
    \label{split1}
    &=
    -\tfrac{1}{(\ell_N+1)(L+1)}
    \sum_{
    \substack{ w\in B_{\ell_N}
    \\z\in B_{L}}
    }
    \int_{\Omega_N}
    \big(
    \eta(w)-\eta(z)
    \big)
    \text{V}_{m,M}(\inner{\eta}_{\ell_N},\inner{\eta}_{L})
    \nabla_{w,z}[\tau^{-x}f](\eta)
    \rmd\nu_\alpha^N(\eta)
\\&
\label{split2}
    +\tfrac{1}{(\ell_N+1)(L+1)}
    \sum_{
    \substack{ w\in B_{\ell_N}
    \\z\in B_{L}}
    }
    \int_{\Omega_N}
    \big(
    \eta(w)-\eta(z)
    \big)
    \text{V}_{m,M}(\inner{\eta}_{\ell_N},\inner{\eta}_{L})
    \big(
    \mathbf{1}+\theta_{w,z}
    \big)[\tau^{-x}f](\eta)
    \rmd\nu_\alpha^N(\eta)
    ,
\end{align}
with $\mathbf{1}\equiv\mathbf{1}_{\Omega_N}$ the identity operator in $\Omega_N$. Let us focus on the second line in the previous display. From a change of variables and the fact that $\alpha$ is a constant profile,
\begin{multline}
    \int_{\Omega_N}
    \big(
    \eta(w)-\eta(z)
    \big)
    \text{V}_{m,M}(\inner{\eta}_{\ell_N},\inner{\eta}_{L})
    \theta_{w,z}f(\tau^{-x}\eta)
    \rmd\nu_\alpha^N(\eta)
\\=
    -\int_{\Omega_N}
    \big(
    \eta(w)-\eta(z)
    \big)
    \text{V}_{m,M}
    \big(
    \inner{\eta}_{\ell_N},\inner{\eta}_{L}
    \big)
    f(\tau^{-x}\eta)
    \mathbf{1}_{\{z,w\in B_L\}}
    \rmd\nu_\alpha^N(\eta)
\\    -\int_{\Omega_N}
    \big(
    \eta(w)-\eta(z)
    \big)
    \text{V}_{m,M}
    \big(
    \inner{\eta}_{\ell_N},\inner{\eta}_{L}
    +
    \tfrac{\eta(w)-\eta(z)}{L+1}
    \big)
    f(\tau^{-x}\eta)
    \mathbf{1}_{\{w\notin B_L\ni z\}}
    \rmd\nu_\alpha^N(\eta)
    .
\end{multline}
This implies that, from the triangle inequality and then \eqref{V-reg}, the absolute value of the quantity in \eqref{split2} can be bounded from above by 
\begin{align}
    \tfrac{1}{L+1}
    \norm{\partial_2
    \text{V}_{m,M}}_{\infty}.
\end{align} 
It is worth noting that it is implicit that $B_L\subset B_{\ell_N}$, which can be assumed since we will perform the limit $L\to\infty$ after $N\to\infty$. 

Now we focus on the term originating from \eqref{split1}. This is analysed with a standard path argument, exchanging the occupation values of the sites $w$ and $z$. First bound from above 
\begin{align}
    \Big|
       \int_0^t
        \frac1N \sum_{x\in\T_N}
        G_s(\tfrac{x}{N})
    \frac{1}{\ell_NL}
    \sum_{ \substack{ w\in B_{\ell_N} \\z\in B_{L}} }
    \int_{\Omega_N}
    \big(
    \eta(w)-\eta(z)
    \big)
    \text{V}_{m,M}(\inner{\eta}_{\ell_N},\inner{\eta}_{L})
    \nabla_{w,z}[\tau^{-x}f](\eta)
    \rmd\nu_\alpha^N(\eta)
    \rmd s\Big|
    \\
    \leq
    c_{T,G}\norm{\text{V}_{m,M}}_{\infty}
    \frac{1}{\ell_NL}
    \sum_{ \substack{ w\in B_{\ell_N} \\z\in B_{L}} }
    \frac1N
        \int_{\Omega_N}
            \sum_{x\in\T_N}
            \big(
            \mathbf{e}_{x+w,x+z}+\mathbf{e}_{x+z,x+w}
            \big)
            \big|
            \nabla_{x+w,x+z}[f]
            \big|
        \rmd\nu_\alpha^N.
\end{align}
Next, proceeding with a classical path argument, exactly as from \eqref{1block:0} to \eqref{1block:path-end}, we see that the right-hand side above is no larger than 
\begin{align}
    2c_{T,G}\norm{\text{V}_{m,M}}_{\infty}
    \Bigg\{
        \frac{\ell_N}{AN}\frac{\mathfrak{D}_N(\sqrt{f})}{N^2}
        +2^2\frac{\ell_NA}{\min\{\mathbf{r}_{\beta,\ell_N}\}}
    \Bigg\}
    .
\end{align}

From the previous display and the analysis of \eqref{split2}, in order to conclude the proof it is enough to solve, for $A$ and $B$, the system
\begin{align}
        \begin{cases}
        2c_{T,G}\norm{\text{V}_{m,M}}_{\infty}\frac{\ell_N}{ANN^2}
        =\frac{T}{2NB},\\
        \frac{\textbf{c}_\alpha}{B}
        =
        c_{T,G}
    2^3\frac{\norm{\text{V}_{m,M}}_{\infty}}{\min\{\mathbf{r}_{\beta,\ell_N}\}}
    \ell_NA.
    \end{cases}
\end{align}

\end{proof}


\appendix

\section{Auxiliary replacements}\label{app:replacement}

The goal of this section is to present the \textit{so-called} One and Two Blocks Estimates -- Corollary \ref{lem:1block} and \ref{lem:2block}, respectively.  

\begin{Lemma}\label{lem:rep}
Fix $L\in\mathbb{N}_+$, $w\in\T_N$ and $y\notin B_{L}(w)$. For each $x\in\T_N$, let $\varphi_{x}:[0,T]\times \mathcal{D}_{\Omega_N}[0,T]\to\mathbb{R}$ be independent of the occupation-values in $\{x+w+r,x+y\}_{r\in B_{L}}$, and such that $\text{c}_{t,\varphi}:=\int_0^t\sup_{x\in\T_N,\eta\in\Omega_N}|\varphi_x(\cdot,s)|\rmd s<\infty $.

For every $t\in [0,T]$ it holds that
    \begin{align}\label{bound:lem1block}
        \mathbb{E}_{\mu_N}
        \bigg[
        \bigg|
        \int_{0}^t
        \frac1N\sum_{x\in\T_N}\varphi_x(s,\eta_{N^2s})
        \big(
        \eta_{N^2s}(x+w)
        -\eta_{N^2s}(x+z)
        \big)
        \rmd s
        \bigg|
        \bigg]
        \leq 
        4\sqrt{2\text{c}_\alpha}c_{T,\varphi}\frac{\rmd(w,z)}{\sqrt{\min\{\mathbf{r}_{\beta,\ell_N}\}TN^2}}
        .
    \end{align}
\end{Lemma}
\begin{proof}
From \cite[Proposition 4.2]{N25}, it is enough to estimate
    \begin{align}\label{1block:var}
        \frac{c_{\alpha}}{A}
        +
        \int_0^t
        \sup_{f}
        \bigg\{
        \int_{\Omega_N}
        \frac1N\sum_{x\in\T}\varphi_x(s,\eta)
        \big(
        \eta(x+w)-\eta(x+z)
        \big)
        f(\eta)\rmd\nu_\alpha^N(\eta)
        -\frac{\mathfrak{D}_{N}(\sqrt{f}|\nu_\alpha^N)}{2N A}
        \bigg\}
        \rmd s
    \end{align}
with the $\sup$ taken over the set of densities with respect to $\nu_\alpha^N$. From \eqref{by-parts}, we can bound from above the integral over $\Omega_N$ in \eqref{1block:var} by 
    \begin{align}\label{1block:0}
        \frac{\text{c}_{s,\varphi}'}{2N}\sum_{x\in\T_N}
        \int_{\Omega_N}
       	(\mathbf{e}_{x+w,x+z}
        +\mathbf{e}_{x+z,x+w})
        |
        \nabla_{x+w,x+z}[f]
        |
        \rmd\nu_\alpha^N,
    \end{align}
where $\text{c}_{s,\varphi}':=\sup_{x\in\T_N}|\varphi_x(\cdot,s)|_{\infty}$. 

We now proceed with a classical path argument. Let $S=2\rmd(w,z)-1$, denote by $\mathbf{1}_{\Omega_N}$ the identity operator in $\Omega_N$, and consider the sequence of operators defined through
\begin{align}
            \begin{cases}
            P_{n+1}
            =\theta_{x_n,x_{n+1}}P_n,
            &0\leq n\leq S-1,\\
            P_0:=\mathbf{1}_{\Omega_N}
            ,&
        \end{cases}
\end{align}
where the ordered sequence of bonds composed by nearest-neighbour sites $(\{x_n,x_{n+1}\})_n$ describes a path exchanging the occupation-value of the sites $w$ and $z$. Concretely, let us suppose that the site $z$ is at the right of $w$. The other case is analogous. Fixing for example, $x_0=x+w$, we have $x_1=x+w+1$ , $x_2=x+w+2,\dots,x_{\text{d}(w,z)}=x+w+(z-w)$ -- in this way changing the occupation-value of the site $x+z$; next, $x_{\text{d}(w,z)+1}=x+z-1,\dots,x_{S-1}=x+w+1$ , $x_S=x+w$ -- in this way switching the occupation-value of the sites $x+w$ and $x+z$, while leaving the rest of the configuration invariant.

With this, we decompose
    \begin{align}
        \nabla_{x+w,x+z}[f](\eta)
        =\sum_{n=0}^{S-1}
        \nabla_{x_{n},x_{n+1}}[f](P_n\eta)
        , 
    \end{align}
apply the inequality $|a-b|\leq \frac{1}{2A}(\sqrt{a}-\sqrt{b})^2+A(a+b)$, that holds for any $A>0$ and $a,b>0$; and, writing $\mathbf{r}_{\beta,\ell_N}^{k,k+1}\equiv\tau^{k}\mathbf{r}_{\beta,\ell_N}$ (resp. $\mathbf{r}_{\beta,\ell_N}^{k,k-1}\equiv\tau^{k-1}\mathbf{r}_{\beta,\ell_N}$) as the hopping rate in a generic node $\{k,k+1\}$ (resp. $\{k-1,k\}$), the quantity in \eqref{1block:0} can be bounded from above by
    \begin{multline}
        \frac{\text{c}_{s,\varphi}'}{AN}
        \sum_{n=0}^{2\rmd(w,z)-1}
        \int_{\Omega_N}
        \sum_{x\in\T_N}
        \mathbf{r}_{\beta,\ell_N}^{x_n,x_{x+1}}(P_n\eta)
        \big|
        \nabla_{x_{n},x_{n+1}}\sqrt{f}(P_n\eta)
        \big|^2
        \rmd\nu_\alpha^N
        +
        4\text{c}_{s,\varphi}'
        \frac{\rmd(w,z)A}{\min\{\mathbf{r}_{\beta,\ell_N}\}}
        \\
        \leq
        2\text{c}_{s,\varphi}'\frac{\rmd(w,z)}{AN}
        \frac{ \mathfrak{D}_N^{}(\sqrt{f}) }{N^2}
        +
        4\text{c}_{s,\varphi}'
        \frac{\rmd(w,z)A}{\min\{\mathbf{r}_{\beta,\ell_N}\}}
    .
    \label{1block:path-end}
    \end{multline}

In this way, \eqref{1block:var} can be estimated from above by
    \begin{align}
        \frac{c_{\alpha}}{B}
        +
        4\text{c}_{T,\varphi}
        \frac{\rmd(w,z)A}{\min\{\mathbf{r}_{\beta,\ell_N}\}}
        +
        \frac{1}{N}\bigg(
        \text{c}_{T,\varphi}\frac{\rmd(x,z)}{AN^2}
        -\frac{T}{2B}
        \bigg)
        \mathfrak{D}_N(\sqrt{f}|\nu_\alpha^N)
        .
    \end{align}
Solving for $A$ and $B$ the system 
\begin{align}
    \begin{cases}
        \frac{c_{\alpha}}{B}
        =
        4\text{c}_{T,\varphi}
        \frac{\rmd(w,z)A}{\min\{\mathbf{r}_{\beta,\ell_N}\}}
        ,\\
        \text{c}_{T,\varphi}\frac{\rmd(x,z)}{AN^2}
        =\frac{T}{2B}
    \end{cases}
\end{align}
concludes the proof.    

\end{proof}

\begin{Cor}[One-block estimate]\label{lem:1block}
Fix $L\in\mathbb{N}_+$, $w\in\T_N$ and $y\notin B_{L}(w)$. For each $x\in\T_N$, let $\varphi_{x}:[0,T]\times \mathcal{D}_{\Omega_N}[0,T]\to\mathbb{R}$ be independent of the occupation-values in $\{x+w+r,x+y\}_{r\in B_{L}}$, and such that $\text{c}_{t,\varphi}:=\int_0^t\sup_{x\in\T_N,\eta\in\Omega_N}|\varphi_x(\cdot,s)|\rmd s<\infty $.

For every $t\in [0,T]$ it holds that
    \begin{align}\label{bound:cor1block}
        \mathbb{E}_{\mu_N}
        \bigg[
        \bigg|
        \int_{0}^t
        \frac{1}{N}\sum_{x\in\T_N}
        \varphi_{x}(s,\eta_{N^2s})
        \big(
        \inner{\tau^{x+w}\eta_{N^2s}}_{L}
        -\eta_{N^2s}(x+y)
        \big)
        \rmd s
        \bigg|
        \bigg]
\\        \leq 
        \frac{4\sqrt{2\text{c}_\alpha}c_{T,\varphi}}{\sqrt{\min\{\mathbf{r}_{\beta,\ell_N}\}TN^2}}
        \bigg(
        \rmd(w,y)+\frac1L\sum_{r\in B_L}\rmd(r,y)
        \bigg)
.
    \end{align}
\end{Cor}
\begin{proof}
Expressing 
    \begin{align}
        \inner{\tau^{x+w}\eta}_{L}
        -\eta(x+y)
        =
        \frac{1}{L}
        \sum_{r\in B_{L}}
        \big\{
        \eta(x+w+r)-\eta(x+y)
        \big\},
    \end{align}
and replacing the right-hand side above inside the expectation in \eqref{bound:cor1block}, then applying the triangle inequality and Lemma \ref{lem:rep} on each term concludes the proof.
\end{proof}

\begin{Cor}[Two-block estimate]\label{lem:2block}
Fix $\ell,L\in\mathbb{N}_+$ such that $\ell>L$, and $y,w\in\T_N$ such that $B_{\ell}(w)\cap B_{L}(y)=\emptyset$ and $y<w$. For each $x\in\T_N$, let $\varphi_{x}:[0,T]\times \mathcal{D}_{\Omega_N}[0,T]$ be independent of the occupation-values in $B_{\ell}(x+w)\cup B_L(x+y)$ and such that $c_{t,\varphi}:=\int_{0}^t|\varphi_x(s,\cdot)|_{\infty}\rmd s<\infty$.

For every $t\in [0,T]$ it holds that
    \begin{multline}
        \mathbb{E}_{\mu_N}
        \bigg[
        \bigg|
        \int_{0}^t
        \frac{1}{N}\sum_{x\in\T_N}
        \varphi_{x}(s,\eta_{N^2s})
        \big(
        \inner{\tau^{x+w}\eta_{N^2s}}_{\ell}-\inner{\tau^{x+y}\eta_{N^2s}}_{L}
        \big)
        \rmd s
        \bigg|
        \bigg]
   \\
   \leq 
    \frac{4\sqrt{2\text{c}_\alpha}c_{T,\varphi}}{\sqrt{\min\{\mathbf{r}_{\beta,\ell_N}\}TN^2}}
        \bigg(
            \frac1L
            +
            \rmd(w,y)
            +
            \ell+L
        \bigg)
        .
    \end{multline}
\end{Cor}
\begin{proof}

Let $\ell$ be divisible by $L$. Otherwise, one can replace $\ell/L$ by $\floor{\ell/L}$ for the rest of the proof, obtaining in the next display an additive error of the order of $1/L$. For any set $A\subset \T_N$, and $r\neq0$, we shall write $\ndef{rA}=\{ra:\;a\in A\}$. Note that $\cup_{p\in L B_{\ell/L}}B_{L}(p)$ forms a covering of $B_{\ell}$ by non-intersecting sets, and clearly, $|L B_{\ell/L}|=\ell/L$. Decomposing then
\begin{align}
    \inner{\tau^{x+w}\eta}_{\ell}-\inner{\tau^{x+y}\eta}_{L}
    =
    \frac{1}{\ell}
        \sum_{\substack{p\in L B_{\ell/L}\\r\in B_{L}}}
    \big\{
    \eta(x+r+p+w)-\eta(x+r+y)
    \big\},
\end{align}
one proceeds analogously to the proof of previous lemma and then apply Lemma \ref{lem:2block}. To conclude the proof, one bounds from above
\begin{align}
    \frac{L}{\ell}
    \sum_{p\in L B_{\ell/L}}
    \rmd(p,y)
    \leq \ell+L
\end{align}
\end{proof}

\section*{Acknowledgements}

L’auteur remercie chaleureusement ``Le Santy" pour son hospitalité.

\bibliographystyle{plain}
\bibliography{02.biblio}

\end{document}